\documentclass[submission]{FPSAC2025}

\usepackage{amsfonts}
\usepackage{bm}
\usepackage{tikz}
\usepackage{booktabs}
\usepackage{multirow}
\usepackage{cancel}

\usepackage{microtype} 
\usepackage{subdepth} 

\newtheorem{theorem}{Theorem}[section] 
\newtheorem{proposition}[theorem]{Proposition} 
\newtheorem{corollary}[theorem]{Corollary} 
\newtheorem{lemma}[theorem]{Lemma} 

\theoremstyle{definition}
 
\newtheorem{example}[theorem]{Example}
\newtheorem{conjecture}[theorem]{Conjecture}

\Crefname{conjecture}{Conjecture}{Conjectures}

\newcommand{\blue}{blue!65!white}
\newcommand{\mydef}[1]{\textcolor{\blue}{\emph{#1}}} 

\newcommand{\inv}{\operatorname{Inv}} 

\newcommand{\cov}[1]{\lessdot_{#1}} 

\newcommand{\Weak}[1]{\operatorname{Weak}({#1})} 

\newcommand{\wo}{{w_\circ}} 
\newcommand{\sq}[1]{{\sf #1}} 
\newcommand{\Q}{\sq{Q}} 
\newcommand{\s}{\sq{s}} 
\newcommand{\sinv}{\overline{\s}} 
\renewcommand{\r}{\sq{r}} 
\newcommand{\R}{\sq{R}} 
\renewcommand{\c}{\sq{c}} 
\renewcommand{\t}{\sq{t}} 
\renewcommand{\u}{\sq{u}} 
\newcommand{\w}{\sq{w}} 
\newcommand{\wrc}[1]{\sq{w}_\mathcal{R}{(#1)}} 
\newcommand{\ls}[1]{\operatorname{\ell}_\mathcal{S}({#1})} 
\newcommand{\lr}[1]{\operatorname{\ell}_\mathcal{R}({#1})} 

\newcommand{\Sort}[1]{\operatorname{Sort}({#1})} 
\NewDocumentCommand\Skipset{O{c}}{{\mathcal{C}_{#1}}} 
\newcommand{\Camb}[1]{\operatorname{Camb}({#1})} 
\newcommand{\Cambsort}[1]{\operatorname{Camb}_\mathrm{Sort}({#1})} 
\newcommand{\Cambclus}[1]{\operatorname{Camb}_\mathrm{SC}({#1})} 
\newcommand{\Cambnc}[1]{\operatorname{Camb}_\mathrm{NC}({#1})} 
\newcommand{\mCambsort}[2]{\operatorname{Camb}^{(#1)}_\mathrm{Sort}({#2})} 
\newcommand{\mCambclus}[2]{\operatorname{Camb}^{(#1)}_\mathrm{SC}({#2})} 
\newcommand{\mCambnc}[2]{\operatorname{Camb}^{(#1)}_\mathrm{NC}({#2})} 
\newcommand{\leqr}{\leq_{\mathcal{R}}} 
\newcommand{\mleq}{\leq_{(m)}} 

\newcommand{\mCamb}[2]{\operatorname{Camb}^{(#1)}({#2})} 
\newcommand{\Abs}[1]{\operatorname{Abs}({#1})} 
\newcommand{\NC}[1]{\operatorname{NC}({#1})} 
\newcommand{\mNC}[2]{\operatorname{NC}^{(#1)}({#2})} 
\newcommand{\NCL}[1]{\operatorname{NCL}({#1})} 
\newcommand{\subS}[1]{\operatorname{SC}_{\mathcal{S}}({#1})} 
\newcommand{\dimr}[2]{\lvert {#1} \rvert_{\leq #2}} 
\newcommand{\Fwith}[1]{\mathcal{F}_{\supset #1}} 
\newcommand{\Fwithout}[1]{\mathcal{F}_{\not\supset #1}} 
\newcommand{\QCamb}[1]{\Q(#1)} 
\newcommand{\QmCamb}[2]{\Q^{(#1)}(#2)} 
\newcommand{\CambInt}[2]{\operatorname{CC}_{#1}(#2)} 

\newcommand\restr[2]{{
  \left.\kern-\nulldelimiterspace 
  #1 
  \vphantom{\big|} 
  \right|^{#2} 
  }}
\newcommand\restri[2]{{
  \left.\kern-\nulldelimiterspace 
  #1 
  \vphantom{\big|} 
  \right|_{#2} 
  }}

\newcommand{\Artinmon}{\boldsymbol{B}^+}
\newcommand{\bwo}{\boldsymbol{w_\circ}} 
\newcommand{\be}{\boldsymbol{e}} 

\newcommand{\bs}{\boldsymbol{s}}

\newcommand{\mhead}{\texorpdfstring{$m$}{m}}

\title{A new definition for \mhead-Cambrian lattices}

\author{Clément Chenevière\thanks{\href{mailto:clement.cheneviere@lisn.fr}{clement.cheneviere@lisn.fr}. Supported by the ANR--FWF project PAGCAP (ANR-21-CE48-0020).}\addressmark{1},
  \and Wenjie Fang\thanks{\href{mailto:wenjie.fang@univ-eiffel.fr}{wenjie.fang@univ-eiffel.fr}. Partially supported by the ANR--FWF project PAGCAP (ANR-21-CE48-0020).}\addressmark{2},
  \and Corentin Henriet\thanks{\href{mailto:henriet@irif.fr}{henriet@irif.fr}.}\addressmark{3}\addressmark{4}}

\address{\addressmark{1}Université Paris-Saclay, CNRS, LISN, Orsay, France \\ \addressmark{2}Univ Gustave Eiffel, CNRS, LIGM, F-77454 Marne-la-Vallée, France \\ \addressmark{3}IRIF, Université Paris-Cité, Paris, France \\ \addressmark{4}DiMaI, Università degli Studi di Firenze, Florence, Italy}

\received{\today}

\abstract{
  The Cambrian lattices, introduced in (Reading, 2006), generalize the Tamari lattice to any choice of Coxeter element in any finite Coxeter group. 
  They are further generalized to the $m$-Cambrian lattices (Stump, Thomas, Williams, 2015). 
  However, their definitions do not provide a practical setup to work with combinatorially. 

  In this paper, we provide a new equivalent definition of the $m$-Cambrian lattices on simple objects called $m$-noncrossing partitions, using a simple and effective comparison criterion. 
  It is obtained by showing that each interval has a unique maximal chain that is \textsf{c}-increasing, which is computed by a greedy algorithm. 
  Our proof is uniform, involving all Coxeter groups and all choices of Coxeter element at the same time.}

\resume{
  Les treillis cambriens, introduits dans (Reading, 2006), généralisent le treillis de Tamari à tout choix d'élément de Coxeter dans tout groupe de Coxeter fini. 
  Ils ont été généralisé en les treillis $m$-cambriens (Stump, Thomas, Williams, 2015), mais ces définitions ne fournissent pas de modèle combinatoire pratique. 

  Dans cet article, nous donnons une nouvelle définition des treillis $m$-cambriens sur des objets simples appelés $m$-partitions non croisées, avec un critère de comparaison simple et efficace. 
  Elle est obtenue en montrant que chaque intervalle a une unique chaîne maximale \textsf{c}-croissante, calculée par un algorithme glouton. 
  Notre preuve est uniforme et implique tous les groupes de Coxeter et tous les choix d'élément de Coxeter à la fois.
}


\usepackage[backend=bibtex]{biblatex}
\begin{document}

\maketitle

\section{Introduction}

  The Tamari lattice, first defined by Tamari~\cite{tamari_bracketings_1962}, enjoys many links to various domains in combinatorics, see~\cite{pons_hdr_2023} and the references therein.
  It can be defined in various perspectives, thus also has many generalizations.
  Among them are the Cambrian lattices proposed by Reading~\cite{reading_cambrian} in the field of Coxeter combinatorics, defined for all finite Coxeter groups, and parameterized by a choice of Coxeter element.
  Cambrian lattices were further generalized by Stump, Thomas and Williams in~\cite{stump_cataland_2018} to the \emph{$m$-Cambrian lattices}, with several equivalent definitions. 
  However, even for linear type~$A$, these descriptions remain difficult to manipulate combinatorially, which hinders our understanding of these lattices.

  In this work, we propose a new definition of the $m$-Cambrian lattice over so-called \emph{$m$-noncrossing partitions}, which are multichains of $m$ noncrossing partitions in the absolute order.
  We define an easy comparison criterion on these simple objects involving only the absolute order and the Cambrian lattice locally, as illustrated in \Cref{fig:scheme}. We prove that it gives an order relation, which is isomorphic to the $m$-Cambrian lattice. 

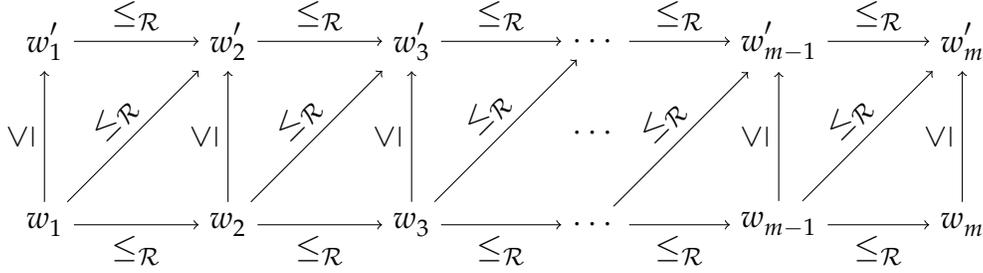
\begin{figure}
  \centering
  \begin{tikzpicture}[scale=1.22]
    \node (W1) at (0, 0) {$w_1$};
    \node (W2) at (2, 0) {$w_2$};
    \node (W3) at (4, 0) {$w_3$};
    \node (W4) at (6, 0) {$\cdots$};
    \node (W5) at (8, 0) {$w_{m-1}$};
    \node (W6) at (10, 0) {$w_m$};
    \draw[->] (W1) -- node [below] {$\leqr$} (W2);
    \draw[->] (W2) -- node [below] {$\leqr$} (W3);
    \draw[->] (W3) -- node [below] {$\leqr$} (W4);
    \draw[->] (W4) -- node [below] {$\leqr$} (W5);
    \draw[->] (W5) -- node [below] {$\leqr$} (W6);
    
    \node (Z1) at (0, 2) {$w'_1$};
    \node (Z2) at (2, 2) {$w'_2$};
    \node (Z3) at (4, 2) {$w'_3$};
    \node (Z4) at (6, 2) {$\cdots$};
    \node (Z5) at (8, 2) {$w'_{m-1}$};
    \node (Z6) at (10, 2) {$w'_m$};
    \node (Mid) at (6, 1) {$\cdots$};
    \draw[->] (Z1) -- node [above] {$\leqr$} (Z2);
    \draw[->] (Z2) -- node [above] {$\leqr$} (Z3);
    \draw[->] (Z3) -- node [above] {$\leqr$} (Z4);
    \draw[->] (Z4) -- node [above] {$\leqr$} (Z5);
    \draw[->] (Z5) -- node [above] {$\leqr$} (Z6);

    \draw[->] (W1) -- node [above, rotate=45] {$\leqr$} (Z2);
    \draw[->] (W2) -- node [above, rotate=45] {$\leqr$} (Z3);
    \draw[->] (W3) -- node [above, rotate=45] {$\leqr$} (Z4);
    \draw[->] (W4) -- node [above, rotate=45] {$\leqr$} (Z5);
    \draw[->] (W5) -- node [above, rotate=45] {$\leqr$} (Z6);

    \draw[->] (W1) -- node [above, rotate=90] {$\leq$} (Z1);
    \draw[->] (W2) -- node [above, rotate=90] {$\leq$} (Z2);
    \draw[->] (W3) -- node [above, rotate=90] {$\leq$} (Z3);
    \draw[->] (W5) -- node [above, rotate=90] {$\leq$} (Z5);
    \draw[->] (W6) -- node [above, rotate=90] {$\leq$} (Z6);
  \end{tikzpicture}
  \caption{Schematic representation of the comparability condition $\mleq$ on $m$-non\-crossing partitions giving an order isomorphic to the $m$-Cambrian lattice.}
  \label{fig:scheme}
\end{figure}

  To show that our new definition of the $m$-Cambrian lattice is equivalent to the original ones in~\cite{stump_cataland_2018}, we resort to another object called \emph{$\c$-increasing chains}, showing that they are canonical certificates of intervals in $m$-Cambrian lattices.
  We also produce a simple greedy algorithm to compute such $\c$-increasing chains.
  To this end, we perform several reductions to simpler cases, involving all parabolic subgroups and all choices of Coxeter elements at the same time. 
  Our proofs, already sketched in~\cite{cheneviere_phd_2023}, rely on several definitions of $m$-Cambrian lattices, especially on noncrossing partitions and subword complexes.

  In the process, we also define a poset structure on intervals of any Cambrian lattice, whose $m$-multichains correspond to intervals in the $m$-Cambrian lattice. 
  This poset is new to our knowledge, even for linear type~$A$, where other poset on Tamari intervals have been studied~\cite{combe_cubical_2023}. 
  It might be a first step to understand the relation conjectured in~\cite{stump_cataland_2018} between $m$-Cambrian intervals for linear type A and intervals in the $m$-Tamari lattice introduced in~\cite{bergeron_higher_2012}, while the later were already enumerated in~\cite{bousquet_number_2011}.

  The rest of this article is organized as follows.
  We lay down the basic definitions in Coxeter combinatorics and of $m$-Cambrian lattices in \Cref{sec:prelim}.
  Then in \Cref{sec:new-def}, we detail our new definition of the $m$-Cambrian lattices and its connection with $\c$-increasing chains.
  We then sketch proof ideas of our results in \Cref{sec:proof-idea}.
  We conclude by \Cref{sec:cambint_poset} with some discussion on the new poset on Cambrian intervals. 

\section{Preliminaries} \label{sec:prelim}

\subsection{Coxeter combinatorics}

  A \mydef{Coxeter group} is a group $W$ with a finite set $\mathcal{S}$ of generators with $s^2 = e$ for $s \in S$, subjected to relations $(st)^{m_{s,t}} = e$ for $m_{s,t} \geq 2$.
  As a convention, the absence of relation between $s, t$ is denoted by $m_{s,t} = \infty$.
  The pair $(W, \mathcal{S})$ is called a \mydef{Coxeter system}, whose \mydef{rank} is defined as $\lvert \mathcal{S} \rvert$.
  The elements of $\mathcal{S}$ are called \mydef{simple reflections}, while the conjugates of elements in $\mathcal{S}$ are called \mydef{reflections}, and we denote by $\mathcal{R} = \{ w^{-1} s w \mid s \in \mathcal{S}, w \in W \}$ the set of reflections.
  In the following, we consider only finite Coxeter systems, whose study can be reduced to irreducible ones that have been fully classified in \cite{coxeter_groups_1935}.
  Given a Coxeter system $(W, \mathcal{S})$ and $J \subset \mathcal{R}$ conjugated to a subset of $\mathcal{S}$, the \mydef{parabolic subgroup} $W_J$ is the subgroup of $W$ generated by $J$. When $J \subset \mathcal{S}$, we refer to $W_J$ as a \mydef{standard parabolic subgroup}. For $J = \mathcal{S} \setminus \{s\}$, we also denote $W_J$ by $W_{\langle s \rangle}$. 

  We can also view a Coxeter system $(W, \mathcal{S})$ as words in the alphabet $\mathcal{S}$ up to equivalence relations.
  A \mydef{reduced word} of $w \in W$ is a shortest word in $\mathcal{S}$ that gives $w$ as product, which may not be unique, and the \mydef{(Coxeter) length} of $w$, denoted by $\ls{w}$, is the length of any reduced word.
  For $w \in W$, a word $\u$ is \mydef{initial} (resp. \mydef{final}) for $w$ if there is some word $\w$ giving $w$ that starts (resp. ends) with $\u$.
  Two words $\u, \w$ are \mydef{commutation equivalent}, denoted by $\u \equiv \w$, if we can go from $\u$ to $\w$ by available commutations of letters in $W$. 
  The perspective of words leads naturally to the \mydef{Artin monoid} $\Artinmon$ associated to $(W, \mathcal{S})$, which is the free monoid generated by $\mathcal{S}$ quotiented by the relations $(st)^{m_{s,t}} = e$.
    Following conventions in~\cite{stump_cataland_2018}, we use \textsf{sans-serif font} for letters and words, in contrast to the equivalent generators and elements of $W$ in normal font, and elements of $\Artinmon$ in \textbf{bold font}.

  Given a Coxeter system $(W, \mathcal{S})$, the \mydef{(right) weak order} on $W$, denoted by $\Weak{W}$, is a partial order on $W$ such that $u \leq w$ if and only if there is some $v \in W$ such that $w = uv$ and $\ls{w} = \ls{u} + \ls{v}$.
  It can also be defined with the \mydef{(left) inversion set} of $w \in W$, denoted by $\inv(w)$, which is the set of reflections $t$ such that $\ls{tw} < \ls{w}$.
  Then $u \leq w$ if and only if $\inv(u) \subseteq \inv(w)$.
  It is well-known that $\Weak{W}$ is a lattice, thus has a maximal element, denoted by $\wo$.

  Replacing the set of simple reflections $\mathcal{S}$ by the set of all reflections $\mathcal{R}$ in the definition of the Coxeter length, we obtain the notion of \mydef{absolute length} of an element $w \in W$, denoted by $\lr{w}$.
  The \mydef{absolute order} $\Abs{W}$ can be defined similarly as the weak order, that is $u \leqr w$ whenever there is $v \in W$ with $w = uv$ and $\lr{w} = \lr{u} + \lr{v}$.

  A \mydef{Coxeter word} $\c$ in $(W, \mathcal{S})$ is a word with each element of $\mathcal{S}$ appearing exactly once.
  A \mydef{Coxeter element} $c$ is the product of letters in a Coxeter word $\c$ in $W$.
  Given a Coxeter element $c$, the interval $[e, c]$ in $\Abs{W}$ is called the \mydef{noncrossing partition lattice}, denoted by $\NCL{W, c}$.
  The elements in $\NCL{W, c}$ are called \mydef{$c$-noncrossing partitions}, and their set is denoted by $\NC{W, c}$.
  It is well-known that the \mydef{Kreweras complement} $cw^{-1}$ of $w \in \NC{W, c}$ is also a $c$-noncrossing partition.
  
\subsection{Cambrian lattice and \mhead-Cambrian lattice}

  The Cambrian lattice, first defined by Reading (\textit{cf.}~\cite{reading_cambrian}), is a generalization of the Tamari lattice to all finite Coxeter groups.
  Given a Coxeter word $\c$, we define the \mydef{$\c$-sorting word} of an element $w \in W$, denoted by $\w(\c)$, to be the leftmost subword of $\c^{\infty}$ that is a reduced word of $w$.
  Let $I_k \subseteq \mathcal{S}$ be the set of letters in $\w(\c)$ from the $k$-th copy of $\c$. The element $w$ is \mydef{$\c$-sortable} if $I_{k+1} \subseteq I_k$ for all $k$, meaning that if a letter $\textsf{s}_i$ does not occur in some copy of $\c$ for $\w(\c)$, then it never occurs afterwards.
  We denote by $\Sort{W, \c}$ the set of $\c$-sortable elements in $W$ with the choice of Coxeter word $\c$.
  The ``sortable version'' of the \mydef{Cambrian lattice}, denoted by $\Cambsort{W, \c}$, is the restriction of the weak order to $\Sort{W, \c}$.
  
  We may also construct Cambrian lattices using subword complexes introduced in~\cite{knutson_subword_2004}.
  Given an $\mathcal{S}$-word $\Q$ of length $p$, an element $w \in W$ and an integer $a$ with $\ls{w} \leq a \leq p$, the \mydef{subword complex} $\subS{\Q, w, a}$ is a simplicial complex on vertex set $\{1, \ldots, p\}$ whose \mydef{facets} are subsets of positions of letters in $\Q$ whose removal gives an $\mathcal{S}$-word of $w$ of length $a$.
  The word $\Q$ is called the \mydef{search word} of $\subS{\Q, w, a}$.
  It is clear that all facets are of size $p - a$.
  Given two adjacent facets $I, J$ of $\subS{\Q, w, a}$, that is, $\lvert I \cap J \rvert = p - a - 1$, there are some $i \neq j$ such that $I \setminus \{i\} = J \setminus \{j\}$.
  Suppose that $i < j$, then we say that $J$ is obtained from $I$ with an \mydef{increasing flip}.
  The \mydef{flip poset} on facets of $\subS{\Q, w, a}$ is then defined as the transitive closure of increasing slips.
  Given a Coxeter word $\c$, we consider the $\c$-sorting word $\sq{\wo}(\c)$ of the maximal element $\wo$ in $\Weak{W}$.
  The ``subword complex version'' of the \mydef{Cambrian lattice}, denoted by $\Cambclus{W, \c}$, is the flip poset of the subword complex $\subS{\c \sq{\wo}(\c), \wo, \ls{\wo}}$.
  The word $\c \sq{\wo}(\c)$ is called the \mydef{Cambrian search word} $\QCamb{W, \c}$. 

  The third construction of Cambrian lattices is based on noncrossing partitions.
  We write the $\c$-sorting word of $\wo$ as $\sq{\wo}(\c) = \s_1 \ldots \s_r$.
  We take $t_i \in W$ the product of the word $\s_1 \ldots \s_{i-1} \s_i \s_{i-1} \ldots \s_1$.
  It is clear that $t_i \in \mathcal{R}$ and $t_i \neq t_j$ for $i \neq j$.
  In fact, every reflection is obtained in this way.
  The total order $\leq_{\c}$ on $\mathcal{R}$ defined by $t_1 <_{\c} t_2 < \cdots <_{\c} t_{r}$ is called the \mydef{reflection order} associated to $\c$.
  We also express this total order as a special $\mathcal{R}$-word $\R(\c) = \t_1 \ldots \t_r$.
  It is known \cite[Theorem~3.5]{athanasiadis_shellability_2007} that each $c$-noncrossing partition $w \in \NC{W, c}$ has a unique reduced $\mathcal{R}$-word, denoted by $\wrc{\c}$, such that $\wrc{\c}$ is a subword of $\R(\c)$.
  We call $\wrc{\c}$ the \mydef{$\c$-increasing word} of $w$.

  By concatenating the $c$-increasing word of $w$ to that of its Kreweras complement $cw^{-1}$, we get a reduced $\mathcal{R}$-word for $c$ that is a subword of $\R(\c)^2$, which is called a \mydef{$1$-factorization} of $\c$.
  Such $1$-factorizations are in bijection with $c$-noncrossing partitions, and can be written in \mydef{colored reflections} as $\r_1^{(0)} \ldots \r_k^{(0)} \r_{k+1}^{(1)} \ldots \r_n^{(1)}$, where the superscript $(0)$ (resp. $(1)$), called \mydef{color}, indicates the letter being in the first (resp. second) copy of $\R(\c)$.
  We extend the reflection order $\leq_{\c}$ to colored reflections, ordering first by color, then by the position of the reflections in $\R(\c)$.
  By regarding $\R(\c)^2$ as two copies of $\R(\c)$ with color $0$ and $1$, we can see a $1$-factorization of $\c$ as a set of colored reflections in $\R(\c)^2$.
  Given a $1$-factorization $\w$ containing some $\r^{(0)}$, as colored reflections in $\w$ come in increasing order with $\leq_{\c}$, we can write $\w = \w' \cdot \r^{(0)} \cdot \w'' \cdot \w'''$, where $\w''$ is the subword of $\w$ with all letters between $\r^{(0)}$ and $r^{(1)}$ in $\leq_{\c}$.
  We define the \mydef{increasing rotation} of $\w$ at the \mydef{rotated reflection} $r^{(0)}$ to be the new $\mathcal{R}$-word $\w' \cdot \overline{\w''} \cdot r^{(1)} \cdot \w'''$, where $\overline{\w''}$ is obtained by conjugating each letter in $\w''$ by $r$.
  By \cite{stump_cataland_2018}, such a word is also a $1$-factorization of $\c$.
  The ``noncrossing version'' of the \mydef{Cambrian lattice} $\Cambnc{W, c}$ is the poset on $1$-factorizations of $\c$ given by the covering relation where the elements covering $\w$ are exactly increasing rotations of $\w$.

\begin{figure}
  \centering
  \begin{minipage}[c]{0.7\linewidth}
    \begin{tabular}{ccccccccc}
      \toprule
      \multirow{2}{*}{Element} & \multirow{2}{*}{$\operatorname{Camb}_{\textrm{Sort}}$} & \multirow{2}{*}{$\operatorname{Camb}_{\textrm{SC}}$} & \multicolumn{6}{c}{$\operatorname{Camb}_{\textrm{NC}}$} \\ 
      & & & $\s$ & $\u$ & $\t$ & $\s$ & $\u$ & $\t$ \\
      \midrule
      1 & $\phantom{\s} \varepsilon$ & $\cancel{\s} \, \cancel{\t} \, \s \, \t \, \s$ & $*$ & $ $ & $*$ & $ $ & $ $ & $ $ \\
      2 & $\s|$                   & $\s \, \cancel{\t} \, \cancel{\s} \, \t \, \s$ & $ $ & $*$ & $ $ & $*$ & $ $ & $ $ \\
      3 & $\s \t | \phantom{\t}$  & $\s \, \t \, \cancel{\s} \, \cancel{\t} \, \s$ & $ $ & $ $ & $*$ & $ $ & $*$ & $ $ \\
      4 & $\t |$                  & $\cancel{\s} \, \t \, \s \, \t \, \cancel{\s}$ & $*$ & $ $ & $ $ & $ $ & $ $ & $*$ \\
      5 & $\s \t | \s$            & $\s \, \t \, \s \, \cancel{\t} \, \cancel{\s}$ & $ $ & $ $ & $ $ & $*$ & $ $ & $*$ \\
      \bottomrule
    \end{tabular}
  \end{minipage}
  \begin{minipage}[c]{0.2\linewidth}
    \begin{tikzpicture}[edge/.style={color=blue!95!black, thick}]
      \node (node_0) at (0,0) { $1$ };
      \node (node_1) at (1.5,1.5) { $4$ };
      \node (node_4) at (0,3) { $5$ };
      \node (node_2) at (-1,1) { $2$ };
      \node (node_3) at (-1,2) { $3$ };
      
      \draw [edge,->] (node_0) -- (node_1);
      \draw [edge,->] (node_0) -- (node_2);
      \draw [edge,->] (node_2) -- (node_3);
      \draw [edge,->] (node_3) -- (node_4);
      \draw [edge,->] (node_1) -- (node_4);
    \end{tikzpicture}
  \end{minipage}
  \caption{Three versions of $\Camb{\mathfrak{S_3}, \s \t}$. Here, $\s = (12)$, $\t = (23)$, $\u = (13)$.}
  \label{fig:cambrian-example}
\end{figure}
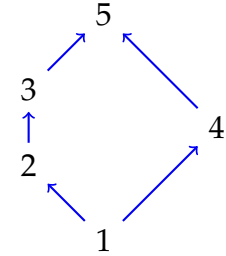

  For a comparison of the three versions of the Cambrian lattice, see \Cref{fig:cambrian-example}. It was shown in~\cite[theorems 5.7.3 and 6.8.6]{stump_cataland_2018} that the three definitions are equivalent.

\begin{proposition} \label{prop:cambrian-equiv}
  We have $\Cambsort{W, \c} \simeq \Cambclus{W, \c} \simeq \Cambnc{W, \c}$ for any Coxeter system $(W, \mathcal{S})$ and any Coxeter word $\c$, with isomorphisms explicitly given.
\end{proposition}

  In \cite{stump_cataland_2018}, Stump, Thomas and Williams further generalized the Cambrian lattices to \emph{$m$-Cambrian lattices}, which also comes in three flavors.
  We see that the notions of weak order and $\c$-sortable elements also apply to the Artin monoid $\Artinmon$.
  Hence, the ``sortable version'' of the \mydef{$m$-Cambrian lattice} $\mCambsort{m}{W, \c}$ is the restriction of the weak order for the Artin monoid to $\c$-sortable elements of the interval $[\be, \bwo^m]$.
  The ``subword complex version'' of the \mydef{$m$-Cambrian lattice} $\mCambclus{m}{W, \c}$ is simply the flip poset of the subword complex $\subS{\c \sq{\wo}^m(\c), \wo^m, m \ls{\wo}}$, with $\wo$ the maximal element of $\Weak{W}$ and $\sq{\wo}(\c)$ its $\c$-sorting word.
  The word $\c \sq{\wo}^m(\c)$ is called the \mydef{$m$-Cambrian search word} $\QmCamb{m}{W, \c}$.
  For the ``noncrossing version'', we generalize the notion of a $1$-factorization of $\c$ to an \mbox{\mydef{$m$-factorization}} of $\c$, which is a subword of $\R(\c)^{m+1}$ that is a reduced $\mathcal{R}$-word of $c$.
  Such an $m$-factorization can be seen as a word of colored reflections with $m + 1$ colors, one for reflections from each copy of $\R(\c)$.
  The notion of increasing rotation can be extended to apply to any colored reflection $r^{(i)}$ of color $i < m$ and turn it into $r^{(i+1)}$ while conjugating all colored reflections in between in an $m$-factorization.
  The transitive closure of increasing rotations gives the ``noncrossing version'' of the \mydef{$m$-Cambrian lattice} $\mCambnc{m}{W, \c}$.
  Again, similar to \Cref{prop:cambrian-equiv}, the three definitions are equivalent.

\begin{proposition}[\cite{stump_cataland_2018}] \label{prop:m-cambrian-equiv}
  We have $\mCambsort{m}{W, \c} \simeq \mCambclus{m}{W, \c} \simeq \mCambnc{m}{W, \c}$ for any Coxeter system $(W, \mathcal{S})$, Coxeter word $\c$, and $m \geq 1$, with isomorphisms explicitly given.
\end{proposition}

\section{A new definition of the \mhead-Cambrian lattice} \label{sec:new-def}

  The $m$-factorizations in $\mCambnc{m}{W, \c}$ are in bijection with $m$-multichains $w_1 \leqr \cdots \leqr w_m$ in $\NCL{W, c}$ that we call \mydef{$m$-noncrossing partitions}: we take $w_i$ as the product of the subword of colored reflections with color at least $m - i + 1$ in a given $m$-factorization.
  We thus propose the following new definition of the $m$-Cambrian lattice.
  Let $\mNC{m}{W, c}$ be the set of $m$-noncrossing partitions.
  For $w_{(m)}, w'_{(m)} \in \mNC{m}{W, c}$ with $w_{(m)} = w_1 \leqr \cdots \leqr w_m$ and $w'_{(m)} = w'_1 \leqr \cdots \leqr w'_m$, we set $w_{(m)} \mleq w'_{(m)}$ if and only if:
  \begin{enumerate}
  \item \label{item:vert-cond} \textbf{Vertical condition}: For all $1 \leq i \leq m$, we have $w_i \leq w'_i$ in $\Cambnc{W, c}$;
  \item \label{item:diag-cond} \textbf{Diagonal condition}: For all $1 \leq i < m$, we have $w_i \leqr w'_{i+1}$ in $\NCL{W, c}$.
  \end{enumerate}
  See \Cref{fig:scheme} for a scheme of the two conditions.
  We claim that \mbox{$(\mNC{m}{W, c}, \mleq)$} is a poset isomorphic to the $m$-Cambrian lattice.
  We note that it is not trivial that our rule gives a partial order, as transitivity cannot be easily proven by composing the conditions.

  \begin{example}\label{ex:m-cambrian_comparison}
    In $A_2$, with $\c = \s \t$ and $m=4$, for $w_a = \varepsilon \leqr s \leqr s \leqr s$, $w_b = sts \leqr sts \leqr st \leqr st$, and $w_c = sts \leqr sts \leqr sts \leqr st$, we have $w_a \mleq w_b$ but $w_a \not\mleq w_c$. 
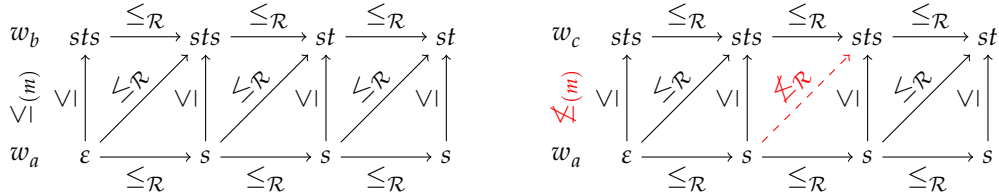
    \begin{center} \unskip
      \begin{tikzpicture}[scale=.8] \footnotesize
        \node (W0) at (-1, 0) {$w_a$};
        \node (W1) at (0, 0) {$\varepsilon$};
        \node (W2) at (2, 0) {$s$};
        \node (W3) at (4, 0) {$s$};
        \node (W4) at (6, 0) {$s$};
        \draw[->] (W1) -- node [below] {$\leqr$} (W2);
        \draw[->] (W2) -- node [below] {$\leqr$} (W3);
        \draw[->] (W3) -- node [below] {$\leqr$} (W4);
        
        \node (Z0) at (-1, 2) {$w_b$};
        \node (Z1) at (0, 2) {$sts$};
        \node (Z2) at (2, 2) {$sts$};
        \node (Z3) at (4, 2) {$st$};
        \node (Z4) at (6, 2) {$st$};
        \draw[->] (Z1) -- node [above] {$\leqr$} (Z2);
        \draw[->] (Z2) -- node [above] {$\leqr$} (Z3);
        \draw[->] (Z3) -- node [above] {$\leqr$} (Z4);
        
        \draw[->] (W1) -- node [above, rotate=45] {$\leqr$} (Z2);
        \draw[->] (W2) -- node [above, rotate=45] {$\leqr$} (Z3);
        \draw[->] (W3) -- node [above, rotate=45] {$\leqr$} (Z4);
        
        \draw[->] (W1) -- node [above, rotate=90] {$\leq$} (Z1);
        \draw[->] (W2) -- node [above, rotate=90] {$\leq$} (Z2);
        \draw[->] (W3) -- node [above, rotate=90] {$\leq$} (Z3);
        \draw[->] (W4) -- node [above, rotate=90] {$\leq$} (Z4);

        \node[rotate=90] at (-1, 1) {$\mleq$};

        \node (Y0) at (8, 0) {$w_a$};
        \node (Y1) at (9, 0) {$\varepsilon$};
        \node (Y2) at (11, 0) {$s$};
        \node (Y3) at (13, 0) {$s$};
        \node (Y4) at (15, 0) {$s$};
        \draw[->] (Y1) -- node [below] {$\leqr$} (Y2);
        \draw[->] (Y2) -- node [below] {$\leqr$} (Y3);
        \draw[->] (Y3) -- node [below] {$\leqr$} (Y4);

        \node (X0) at (8, 2) {$w_c$};
        \node (X1) at (9, 2) {$sts$};
        \node (X2) at (11, 2) {$sts$};
        \node (X3) at (13, 2) {$sts$};
        \node (X4) at (15, 2) {$st$};
        \draw[->] (X1) -- node [above] {$\leqr$} (X2);
        \draw[->] (X2) -- node [above] {$\leqr$} (X3);
        \draw[->] (X3) -- node [above] {$\leqr$} (X4);

        \draw[->] (Y1) -- node [above, rotate=45] {$\leqr$} (X2);
        \draw[red,->,dashed] (Y2) -- node [above, rotate=45] {$\not\leqr$} (X3);
        \draw[->] (Y3) -- node [above, rotate=45] {$\leqr$} (X4);

        \draw[->] (Y1) -- node [above, rotate=90] {$\leq$} (X1);
        \draw[->] (Y2) -- node [above, rotate=90] {$\leq$} (X2);
        \draw[->] (Y3) -- node [above, rotate=90] {$\leq$} (X3);
        \draw[->] (Y4) -- node [above, rotate=90] {$\leq$} (X4);

        \node[red,rotate=90] at (8, 1) {$\not\mleq$};
      \end{tikzpicture}
    \end{center}
  \end{example}

  To show that our new definition is equivalent to the ones in~\cite{stump_cataland_2018}, we introduce the notion of a \mydef{$\c$-increasing chain} between two $m$-factorizations $\w, \w' \in \mCambnc{m}{W, \c}$, whose rotated reflections are increasing in the (extended) reflection order $\leq_\c$.
  
\section{Proof ideas} \label{sec:proof-idea}

  In this section, we sketch the main steps of the proofs of our results.
  The first step is to show the existence and uniqueness of a $\c$-increasing chain between two comparable elements.
  As a tool, we give a greedy algorithm to compute such a chain, and thus decide comparability in the $m$-Cambrian lattice.
  Finally, we show that our comparison rule of $m$-noncrossing partitions in~\Cref{sec:new-def} is equivalent to the existence of a $\c$-increasing chain.

\subsection{The greedy algorithm}

  Given a $\c$-increasing chain from $\w$ to $\w'$, we prove that its first rotated reflection must be the smallest element of $\w$ that is not in $\w'$.
  This observation leads to the following greedy algorithm to compute a $\c$-increasing chain from $\w$ to $\w'$.

  \begin{proposition}\label{prop:greedy_algorithm}
    Let $\w$ and $\w'$ be two $m$-factorizations of $\c$.
    We define $\w_0 = \w, \w_1, \ldots, \w_{(m+1)|\mathcal{R}|}$, where $\w_{j+1}$ is computed from $\w_j$, $\w'$, and the $j$-th letter $r_j$ of $\R(\c)^{m+1}$. We have
    \begin{itemize}
      \item $\w_{j+1} = \w_j$ if $r_j \notin \w_j \setminus \w'$ and $r_j \notin \w' \setminus \w_j$; 
      \item $\w_{j+1}$ is the increasing rotation of $\w_j$ at $r_j$ if $r_j \in \w_j \setminus \w'$ and $r_j$ is not of color $m$;
      \item Otherwise, $\w_{j+1}$ is undefined, and the computation aborts.
    \end{itemize}
    We get the $\c$-increasing chain from $\w$ to $\w'$, if it exists, and otherwise a certificate for non-existence.
  \end{proposition}

  As a technical tool for the proofs, for an $m$-factorization $\w$ and a colored reflection $r$, we define $\dimr{\w}{r}$ to be the number of colored reflections of $\w$ weakly smaller than $r$ in $\leqr$. 
  The following lemma recalls how $\dimr{\w}{r}$ is changed by an increasing rotation.
  
  \begin{lemma} \label{lem:root_conf_vect}
    Let $\w$ be an $m$-factorization of $\c$, and $r$ be a colored reflection.    
    If $\w'$ is an upper cover of $\w$, then $\dimr{\w'}{r} = \dimr{\w}{r} - 1$ if $r(\w, \w') = r$, and $\dimr{\w'}{r} = \dimr{\w}{r}$ if $r(\w, \w') >_\c r$.
  \end{lemma}

  The case $r(\w,\w') <_\c r$ is not tackled here as it is not needed in the proof of~\Cref{lem:smallest_flip}.

  \begin{lemma}\label{lem:smallest_flip}
    Let $\w = \w_0 \cov{} \w_1 \cov{} \cdots \cov{} \w_k = \w'$ be any saturated chain from $\w$ to $\w'$. Let $r$ be the smallest element of $\w \setminus \w'$ for $\leq_\c$. 
    The smallest rotation reflection for $\leq_\c$ in the chain is equal to $r$.
  \end{lemma}

  \begin{proof}
    Let $r' = r(\w_j, \w_{j+1})$ be the smallest rotation reflection of the chain, which means $r(\w_{k}, \w_{k+1}) >_\c r'$ for $k \neq j$.
    We prove that $r' = r$ by contradiction.

    If $r' < r$, then by the definition of $r$, we have $\dimr{\w'}{r'} = \dimr{\w}{r'}$, as $\w'$ and $\w$ have the same reflections strictly smaller than $r$.
    By~\Cref{lem:root_conf_vect}, we have $\dimr{\w'}{r'} = \dimr{\w_k}{r'} \leq \cdots \dimr{\w_{j+1}}{r'} < \dimr{\w_j}{r'} \leq \cdots \leq \dimr{\w_0}{r'} = \dimr{\w}{r'}$, which is a contradiction. 

    If $r' > r$, by the definition of $r$, we have $\dimr{\w'}{r} < \dimr{\w}{r}$. 
    By~\Cref{lem:root_conf_vect}, for all $i$, as $r(\w_i, \w_{i+1}) > r$, we have $\dimr{\w_{i+1}}{r} = \dimr{\w_{i}}{r}$, hence $\dimr{\w'}{r} = \dimr{\w}{r}$, which is impossible.
  \end{proof}

  \begin{corollary}\label{cor:unique_increasing_chain}
    There exists at most one $\c$-increasing saturated chain between two $m$-factorizations.
  \end{corollary}

  \begin{proof}
    The smallest rotated reflection of a $\c$-increasing chain from $\w$ to $\w'$ is its first, which is the smallest one in $\w \setminus \w'$ by~\Cref{lem:smallest_flip}. 
    We conclude by induction on chain length.
  \end{proof}

  Previous discussion shows that if a $\c$-increasing chain from $\w$ to $\w'$ exists, then we can get it by the greedy algorithm of~\Cref{prop:greedy_algorithm}.
  It remains to show that such a chain always exists between two comparable elements.
  To prove it, we take any saturated chain from $\w$ to $\w'$ and make it $\c$-increasing by successive local changes. 
  The following key lemma, proven in~\Cref{ssec:local_reordering_lemma}, explains how to rectify reflection order locally in a chain.

  \begin{lemma}[\textbf{Local reordering lemma}]\label{lem:decreasing_two_chain}
    Let $\w_0 \cov{} \w_1 \cov{} \w_2$ in $\mCamb{m}{W,\c}$ with $r(\w_0, \w_1) = r >_\c r' = r(\w_1, \w_2)$.
    Then $r' \in \w_0$ and setting $\w'_1$ the upper cover of $\w_0$ such that $r(\w_0, \w'_1) = r'$, we have $\w'_1 \leq \w_2$.
    In particular, there is a saturated chain from $\w_0$ to $\w_2$ that starts with $\w_0 \cov{} \w'_1$.
  \end{lemma}


  \begin{proposition}\label{thm:c_increasing_chain}
    If $\w \leq \w'$ in $\mCamb{m}{W,\c}$, there exists a $\c$-increasing chain from $\w$ to $\w'$, and it is of maximal length among chains from $\w$ to $\w'$. 
  \end{proposition}

  \begin{proof}
    Suppose $\w \leq \w'$ in the $m$-Cambrian lattice.
    Then, there exists a saturated chain $\w = \w_0 \cov{} \w_1 \cov{} \cdots \cov{} \w_k = \w'$ which may not necessarily be $\c$-increasing.

    Let $i$ be the largest index such that the $i$ reflection rotations of the subchain from $\w_0$ to $\w_i$ are the smallest for $\leq_\c$ of the whole chain and appear in order.
    Let $j > i$ be the index of the next smallest flip reflection $r(\w_j, \w_{j+1})$ of the subchain from $\w_i$ to $\w_k$.

    If $i = k$, the chain is $\c$-increasing and we are done. Otherwise, $\w_{j-1} \cov{} \w_j \cov{} \w_{j+1}$ respects the conditions of~\Cref{lem:decreasing_two_chain}. 
    There is thus a chain $\w_{j-1} \cov{} \w'_j \cov{} \cdots \cov{} \w_{j+1}$ of length at least $2$ where $r(\w_{j-1}, \w'_j) = r(\w_j, \w_{j+1})$. 
    Replacing $\w_{j-1} \cov{} \w_j \cov{} \w_{j+1}$ by this new subchain in the saturated chain from $\w_0$ to $\w_k$, by \Cref{lem:smallest_flip}, the smallest rotation reflection of the new subchain from $\w_{i}$ to $\w_k$ is $r(\w_{j-1}, \w'_j)$, so either $j-i$ has decreased by one, or $i$ has increased by at least one.
    Iterating the process, $i$ eventually grows.
    As chain length is bounded, we end up with a $\c$-increasing chain.
    By unicity of such a chain (see~\Cref{cor:unique_increasing_chain}), its length is maximal among chains from $\w$ to $\w'$.
  \end{proof}

  \begin{theorem}\label{thm:greedy_comparison}
    We have $\w \leq \w'$ in $\mCamb{m}{W,\c}$ if and only if there is a $\c$-increasing chain from $\w$ to $\w'$. Moreover, it can be decided by the greedy algorithm in~\Cref{prop:greedy_algorithm}.
  \end{theorem}

  \subsection{The local reordering lemma}\label{ssec:local_reordering_lemma}
  It remains to prove the local reordering lemma~(\Cref{lem:decreasing_two_chain}).
  The main idea is to reduce it to the case of a rank $2$ Coxeter group, then to a $1$-Cambrian lattice in this case, using the subword complex version of $m$-Cambrian lattices. 
  The main technical tools are the shift operator from~\cite{stump_cataland_2018}, and the restriction of $m$-Cambrian lattices to a parabolic subgroup. 

  \begin{proposition}[{\cite[Section~5.4.1]{stump_cataland_2018}}]\label{prop:shifted_cambrian}
    Let $s$ be the generator corresponding to the initial letter of a Coxeter word $\c$ in a Coxeter group $W$. 
    Let $\c' = \sinv \c \s$ be the Coxeter word of $W$ obtained by moving the first letter of $\c$ to the end. 
    Let $\s'$ be the generator obtained by conjugating $s$ by $\wo^m$. 
    Then the $\mathcal{S}$-word $\sinv \c \sq\wo^m(\c) \s'$ is commutation equivalent to the $m$-Cambrian search word $\QmCamb{m}{W, \c'}$. Hence, the flip poset of $\subS{\sinv \c \sq\wo^m(\c) \s', \wo^m, m \ls{\wo}}$ is isomorphic to $\mCambclus{m}{W,\c'}$.
  \end{proposition}

  This gives a canonical bijection between the facets of $\mCambclus{m}{W,\c}$ and those of $\mCambclus{m}{W, \c'}$, which is called the \mydef{shift operator} over $s$. 
  It is \emph{not} a poset isomorphism, but we can split $\mCambclus{m}{W,\c}$ into two intervals such that it is in each interval.

  \begin{proposition}\label{prop:shift_first_letter_123}
    Let $s$ be the generator corresponding to the initial letter of a Coxeter word $\c$ in a Coxeter group $W$. Let $\Fwith{1}$ (resp.\ $\Fwithout{1}$) be the set of facets of $\mCambclus{m}{W, \c}$ in which the first letter of $\QmCamb{m}{W, \c}$ is present (resp. absent).
    The two subsets $\Fwith{1}$ and $\Fwithout{1}$ are two intervals of $\mCambclus{m}{W, \c}$, on which the restriction of the shift operator is an isomorphism of posets.
  \end{proposition}

  \begin{proof}
    The facets of $\Fwith{1}$ are $\c$-sortable elements with a $\c$-sorting word not starting by $\s$.
    By definition of $\c$-sortability, elements of $\Fwith{1}$ live in $W_{\langle s \rangle}$ and are $\sinv \c$-sortable.
    The restriction of $\mCamb{m}{W, \c}$ to $\Fwith{1}$ is isomorphic to $\mCamb{m}{W_{\langle s \rangle}, \sinv \c}$.

    The subset $\Fwithout{1}$ is formed by $\c$-sortable elements whose $\c$-sorting word starts with $\s$. 
    Hence, $\Fwithout{1}$ is the set of $\c$-sortable elements in $[\bs, \bwo^m]$, thus an interval in $\mCambsort{m}{W, \c}$.

    By~\Cref{prop:shifted_cambrian}, the shift operator is a bijection between facets of the two subword complexes.
    Within $\Fwith{1}$ and $\Fwithout{1}$, covering relations never involve the first letter, hence are preserved by the shift operator.
    On $\Fwith{1}$ and $\Fwithout{1}$, it is an isomorphism of posets.
  \end{proof}

  We can describe precisely the isomorphism between $\Fwith{1}$ and $\mCamb{m}{W_{\langle s \rangle}, \sinv \c}$.

  \begin{proposition}\label{prop:shift_first_letter_4}
    In the situation of \Cref{prop:shift_first_letter_123}, there is a subset $X$ of positions in the search word for $\mCambclus{m}{W, c}$ such that facets of $\Fwith{1}$ are included in $\{1\} \cup X$ and vice-versa.
    In terms of $m$-reflections, the isomorphism between $\Fwith{1}$ and $\mCamb{m}{W_{\langle s \rangle}, \sinv \c}$ is given by removing the reflection $s^{(0)}$ from the facets.
  \end{proposition}

  Positions in $X$ are said to be \mydef{compatible} with the first letter of the search word.

  \begin{proof}
    Let $\Q = \QmCamb{m}{W, \c}$ and $\Q' = \QmCamb{m}{W_{\langle s \rangle}, \sinv \c}$.
    We identify letters of $\Q'$ with a subset $X$ of positions in $\Q$.
    The smallest facet in $\mCambclus{m}{W, \c}$ is the initial copy of $\c$. 
    We identify letters of $\c$ not equal to $\s$ with the corresponding letter in the initial copy of $\sinv \c$ in $\QmCamb{m}{W_{\langle s \rangle}, \sinv \c}$.
    Since the two simplicial sets are isomorphic, we can identify increasing flips, which allows us to propagate the identification to all letters of $\Q'$.
    All facets of $\Fwith{1}$ live within $\{1\} \cup X$ by construction. 
    The identification of simplicial complexes gives the reverse direction.
    Finally, any $m$-factorization of $\sinv \c$ can be completed into one for $\c$ by adding $s^{(0)}$, which means we can turn a facet for $\Q'$ into one for $\Q$.
  \end{proof}

  The shift operator allows us to move the first letter of the subword complex of the $m$-Cambrian lattice to the end.
  Combining it with~\Cref{prop:shift_first_letter_4}, we can fix any subset $E$ of compatible positions in the search word, and reduce the rank by the size of $E$.
  This is formalized in the next proposition, which generalizes~\cite[Proposition~3.6]{pilaud_brick_2015}.

  \begin{proposition}\label{prop:rank_reduction}
    Let $W$ be a Coxeter group of rank $n$, $I$ a facet of $\mCambclus{m}{W, \c}$, and $E$ a subset of $I$. 
    The set $\Fwith{E}$ of facets of $\mCambclus{m}{W, \c}$ containing $E$ forms an interval isomorphic to an $m$-Cambrian lattice of the parabolic subgroup $W'$ of rank $n-\lvert E \rvert$ generated by $I \setminus E$.
  \end{proposition}

  \begin{proof}
    We proceed by induction on $|E|$.
    It suffices to prove it for $\lvert E \rvert = 1$.
    Let $\Q = \s_1 \ldots \s_p$ be the $m$-Cambrian search word $\QmCamb{m}{W, \c}$.
    The case $E = \{ 1 \}$ is done in~\Cref{prop:shift_first_letter_4}.

    Let $E = \{ i+1 \}$ for some $i > 0$. 
    Let $\Q' = \s_{i+1} \ldots \s_p \s'_1 \ldots \s'_{i}$ be the search word obtained by applying the shift operator $i$ times.
    By~\Cref{prop:shifted_cambrian}, $\Q' \equiv \QmCamb{m}{W, \c'}$ for some Coxeter word $\c'$.
    Now $\s_{i+1}$ is the initial letter of $\Q'$, and $\Fwith{E}$ is an interval isomorphic to $\mCamb{m}{W_{\langle \s_{i+1} \rangle}, \overline{\s_{i+1}}\c'}$, of rank $n-1$.

    To complete the proof, we need to shift back the $i$ last letters of $\Q'$ to the beginning.
    By~\Cref{prop:shift_first_letter_4}, letters of $\Q'' = \QmCamb{m}{W_{\langle \s_{i+1} \rangle}, \overline{\s_{i+1}}\c'}$ are identified with a subset $X$ of positions in $\Q'$ compatible with $\s_{i+1}$.
    When shifting back a letter $\s_k$, if $k \in X$, we shift back the corresponding letter in $\Q''$; if $k \not \in X$, the parabolic subword complex remains the same but the parabolic subgroup is conjugated by $s_k$.
    We have finally that $\Fwith{E}$ is an interval in $\mCamb{m}{W, \c}$ isomorphic to an $m$-Cambrian lattice in a conjugate of $W_{\langle \s_{i+1} \rangle}$.
  \end{proof}
   
  We can now prove the local reordering lemma.

  \begin{proof}[Proof of the local reordering lemma~(\Cref{lem:decreasing_two_chain}).]
    Let $\w_0 \cov{} \w_1 \cov{} \w_2$ in $\mCamb{m}{W,\c}$ be such that $r(\w_0, \w_1) = r >_\c r' = r(\w_1, \w_2)$.
    We prove that $r' \in \w_0$, and that $\w'_1 \leq \w_2$ where $\w'_1$ is the upper cover of $\w_0$ such that $r(\w_0, \w'_1) = r'$.

    First, the rotation from $\w_0$ to $\w_1$ does not modify the reflections smaller than $r$.
    Thus, $r'$ is already in $\w_0$, meaning that $\w'_1$ is well-defined, and $\w'_1 \neq \w_2$. 
    As $\w_2$ is obtained from $\w_0$ by two successive rotations, the two corresponding facets of the subword complex contain at least $n-2$ common positions. 
    By~\Cref{prop:rank_reduction}, $\w_0$ and $\w_2$ live in an $m$-Cambrian lattice of rank $2$, say $\mCamb{m}{I_2(k), \s \t}$ for some $k$ and some Coxeter word $\s\t$.
    
    Without loss of generality, thanks to the shift operators, we may assume that $r'$ is the smallest reflection $s^{(0)}$ in the reflection order associated to $\s\t$.
    If $r$ is not of color $0$, then the two increasing rotations commute, and $\w'_1 \leq \w_2$.
    Otherwise, $[\w_0, \w_2]$ is isomorphic to an interval in the $1$-Cambrian lattice $\Camb{I_2(k), \s\t}$, and the result is easily checked since $\w_0$ and $\w_2$ must be the bottom and top elements of the poset. Hence, $\w'_1 \leq \w_2$.
  \end{proof}

  \subsection{The comparison criterion}
  The greedy algorithm in~\Cref{prop:greedy_algorithm} provides a way to decide comparability in the $m$-Cambrian lattice. 
  A reinterpretation of the algorithm leads to a criterion for comparability of two $m$-noncrossing partitions.
  The idea is that, instead of processing letters of $\R(\c)^{m+1}$ in one go, we take a ``snapshot'' of the state of the algorithm after each copy of $\R(\c)$.

  Recall that there is a natural bijection between $m$-factorizations of $\c$ and $m$-noncrossing partitions. 
  For $m = 1$, we get a description of the Cambrian lattice on the set of non\-crossing partitions $\NC{W, c}$, with the order relation denoted by $\leq$. 
  The noncrossing partition lattice is another partial order on $\NC{W, c}$, with the order relation denoted by $\leqr$.
  We can in fact define the $m$-Cambrian lattice using $\leq$ and $\leqr$ on $\NC{W, c}$, as described in \Cref{sec:new-def}. This criterion is illustrated in~\Cref{fig:scheme}.

  \begin{proposition}\label{prop:embedding_projecting_cambrian}
    For any $m' < m$, the $m'$-Cambrian lattice $\mCamb{m'}{W, \c}$ embeds in $\mCamb{m}{W, \c}$ as the interval of $m$-factorizations using only $m'+1$ consecutive colors.

    Conversely, the $m$-Cambrian lattice $\mCamb{m}{W, \c}$ projects onto the $m'$-Cambrian lattice $\mCamb{m'}{W, \c}$ by keeping only $m'$ consecutive components of the $m$-factorizations.
  \end{proposition}

  \begin{proof}
    For any $i \leq m-m'$, the $m'$-Cambrian lattice $\mCamb{m'}{W, \c}$ is isomorphic to the interval $[\bwo^{i}, \bwo^{i+m'}]$ in $\mCamb{m}{W, \c}$.
    Conversely, for $\w$ an $m$-factorization of $\c$, we can chase its reflections of color $0$ by rotating at them in the order $\leq_{\c}$. 
    This only modifies the last component of the corresponding $m$-noncrossing partition, and respects the order thanks to~\Cref{thm:greedy_comparison}, so that we can erase its last component.
    A similar argument applies for the erasure of the first component.
  \end{proof}


  Hence, two comparable $m$-noncrossing partitions satisfy the vertical condition~(\Cref{item:vert-cond}).

  \begin{corollary}\label{cor:m_cambrian_proj}
    Let $w_{(m)} = (w_i)_{1 \leq i \leq m}$ and $w'_{(m)} = (w'_i)_{1 \leq i \leq m}$ be two $m$-noncrossing partitions such that $w_{(m)} \leq w'_{(m)}$ in $\mCamb{m}{W, \c}$.
    For all $1 \leq i \leq m$, $w_i \leq w'_i$ in $\Cambnc{W, \c}$.
  \end{corollary}

  \begin{theorem}\label{thm:comparison_criterion}
    Two $m$-noncrossing partitions $\w$ and $\w'$ are comparable in the $m$-Cambrian lattice if and only if they satisfy the vertical and diagonal conditions (\Cref{item:vert-cond,item:diag-cond} in~\Cref{sec:new-def}).
  \end{theorem}

  \begin{proof}
    For $m = 1$, the proof is immediate. Suppose $m > 1$. 
    Let $\w \leq \w'$ in $\mCambnc{m}{W, \c}$, and $w_{(m)}$ and $w'_{(m)}$ be the corresponding $m$-noncrossing partitions. 
    By~\Cref{cor:m_cambrian_proj}, the vertical condition is satisfied.
    During the greedy algorithm in~\Cref{prop:greedy_algorithm} that turns $\w$ into $\w'$, when we read the $i$-th copy of $\R(\c)$, we turn $w_{m-i}$ into $w'_{m-i}$ without changing the other components.
    Hence, we have the diagonal condition $w_{m-i-1} \leq_\mathcal{R} w'_{m-i}$ for all $i$.

    Reciprocally, we start with two $m$-noncrossing partitions $w_{(m)}$ and $w'_{(m)}$ such that $w_i \leq  w'_i$ for all $i$, and $w_i \leq_\mathcal{R} w'_{i+1}$ for all $1 \leq i < m$. 
    Since $w_{m-i} \leq_\mathcal{R} w'_{m-i+1}$, we take $w^{[i]} \in \mNC{m}{W, \c}$ to be $w_1 \leq_\mathcal{R} \cdots \leq_\mathcal{R} w_{m-i} \leq_\mathcal{R} w'_{m-i+1} \leq_\mathcal{R} \cdots  \leq_\mathcal{R} w'_m$. 
    We observe that $w^{[i]}$ and $w^{[i+1]}$ only differ on reflections of color $i$ and $i+1$.
    We use~\Cref{prop:rank_reduction} to remove the other common reflections.
    Since $w_i \leq w'_i$ in $\Camb{W, \c}$, we have $w^{[i]} \leq w^{[i+1]}$ in $\mCamb{m}{W, \c}$, which means $w_{(m)} = w^{[0]} \leq w^{[m]} = w'_{(m)}$.
  \end{proof}

\section{A poset on Cambrian intervals} \label{sec:cambint_poset}

  Inspired by our new definition, we define a binary relation $\preccurlyeq$ on intervals of the Cambrian lattice $\Cambnc{W, \c}$ where $[\u, \u'] \preccurlyeq [\w, \w']$ if and only if $\u \leq_\mathcal{R} \w$, $\u \leq_\mathcal{R} \w'$, and $\u' \leq_\mathcal{R} \w'$.
  Note that transitivity of this relation is here easily proven.
  However, reflexivity fails for certain elements that we call \mydef{red elements}.
  Enforcing reflexivity extends $\preccurlyeq$ to a poset $\CambInt{2}{W, \c}$.
  This construction can be generalized to $k$-multichains of $\Cambnc{W, \c}$  with horizontal and diagonal comparisons, leading to a poset $\CambInt{k}{W, \c}$ with red elements.

\begin{proposition}\label{prop:multichains_in_cambint}
  For any $k$ and $m$, the $m$-multichains of $\CambInt{k}{W, \c}$ are in bijection with $k$-multichains in $\mCamb{m}{W, \c}$, where repetition of red elements is forbidden.
\end{proposition}

  For a finite poset $P$, its Zeta polynomial $Z_P$ counts its $m$-multichains when evaluated at $m+1$.
  It is not applicable here due to red elements, but there still seems to be a polynomial that enumerates multichains in $\CambInt{2}{W, \c}$ with no repeated red element.

  Recall that the Cambrian lattice specializes to the Tamari lattice in the linear type $A$ case. 
  Even in this context, the poset on Tamari intervals we define here seems new to our knowledge, see~\cite{combe_cubical_2023} for another poset structure on Tamari intervals.
  There is another generalization of the Tamari lattice called the $m$-Tamari lattice~\cite{bergeron_higher_2012}, different from the linear type $A$ $m$-Cambrian lattice.
  We conclude by mentioning a conjecture in~\cite[Section~6.10]{stump_cataland_2018}. 

\begin{conjecture}[{\cite{stump_cataland_2018}}]
  The $m$-Tamari lattice and the linear type $A$ $m$-Cambrian lattice have the same number of intervals.
\end{conjecture}

\printbibliography

\end{document}